\documentclass[12pt]{amsart}
\usepackage{amsmath,latexsym,amscd,amsbsy,amssymb,amsfonts,amsthm,fleqn,leqno,
euscript, graphicx, texdraw, pb-diagram}

\numberwithin{equation}{section}
\newtheorem{thm}{Theorem}[section]
\newtheorem{pro}[thm]{Proposition}

\newtheorem{cor}[thm]{Corollary}

\newtheorem{qu}[thm]{Question}

\theoremstyle{definition}
\newtheorem{dfn}[thm]{Definition}

\theoremstyle{remark}

\begin{document}
%%%%%%%%%%% Begin Topmatter %%%%%%%%%%%%%%%%%

\title[Alexandroff type manifolds and homology manifolds]
{Alexandroff type manifolds and homology manifolds}

\author{V. Todorov}
\address{Department of Mathematics, UACG, $1$ H. Smirnenski blvd.,
$1046$ Sofia, Bulgaria} \email{vtt-fte@uacg.bg}

\author{V. Valov}
\address{Department of Computer Science and Mathematics,
Nipissing University, 100 College Drive, P.O. Box 5002, North Bay,
ON, P1B 8L7, Canada} \email{veskov@nipissingu.ca}

\thanks{The second author was partially supported by NSERC
Grant 261914-13.}

 \keywords{Cantor $n$-manifold, cohomological dimension, cohomology groups,
homology $n$-manifold, $K^n_G$-manifold, $V^n$-continuum}

\subjclass[2010]{Primary 55M10, 55M15; Secondary 54F45, 54C55}
\begin{abstract}
We introduce and investigate the notion of (strong) $K^n_G$-manifolds, where $G$ is an abelian group. One of the result related to that notion (Theorem 3.4) implies the following partial answer  to the Bing-Borsuk problem \cite{bb}, whether any partition of a homogeneous metric $ANR$-space $X$ of dimension $n$ is cyclic in dimension $n-1$: If $X$ is a homogeneous metric $ANR$-continuum with $\dim_GX=n$ and $H^{n}(X;G)\neq 0$, then $H^{n-1}(M;G)\neq 0$ for every set $M\subset X$, which cuts $X$ between two disjoint open subsets of $X$. Another implication of Theorem 3.4 (Corollary 3.6) provides an
analog of the classical result of Mazurkiewicz \cite{ma} that no region in $\mathbb R^n$ can be cut by a subset of dimension $\leq n-2$.
Concerning homology manifolds, it is shown that
 if $X$ is arcwise connected complete metric space which is either  a homology $n$-manifold over a group $G$ or a product of at least $n$  metric spaces, then $X$ is a Mazurkiewicz arc $n$-manifold. We also introduce a property which guarantees that $H_k(X,X\setminus x;G)=0$ for every $x\in X$ and $k\leq n-1$, where $X$ is a homogeneous locally compact metric $ANR$.
\end{abstract}
\maketitle\markboth{}{Generalized manifolds}
%%%%%%%%%% End topmatter %%%%%%%%%%%%%%%%%%%%%

%%%%%%%%%%%%%%%%%%%%%%%%%%%%%%%%%%%%%%%%%%%%%%%%%%%%%%%%%%%%%%%%
%%%%%%%%%%%%%%%%%%%%%%%%%%%%%%%%%%%%%%%%%%%%%%%%%%%%%%%%%%%%%%%%

%%%%%%%%%%%%%%%%%% TABLE OF CONTENT %%%%%%%%%%%%%%%%%%%%%%%%%%%%%%%%%%

%\tableofcontents

%%%%%%%%%%%%%%%%%%%%%%%%%%%%%%%%%%%%%%%%%%%%%%%%%%%%%%%%%%%%%%%%%%%
%%%%%%%%%%%%%%%%%%%%%%%%%%%%%%%%%%%%%%%%%%%%%%%%%%%%%%%%%%%%%%%%%%%%%%

\section{Introduction}

In this paper we investigate some properties of generalized Cantor manifolds. Cantor  manifolds were introduced by Urysohn \cite{u} as a generalization of Euclidean manifolds. It appeared that Euclidean manifolds have a richer structure and that was a motivation for the study of further specifications of Cantor manifolds. Another interesting fact is that homogeneous metric $ANR$s have some common properties with generalized Cantor manifolds (see \cite{kktv}, \cite{kv}, \cite{kru93}), and the same is true for homology manifolds (see Corollary 4.2 below). This is not so surprising having in mind the Bing-Borsuk conjecture that every homogeneous separable locally compact metric $ANR$ of dimension $n$ is a homology $n$-manifold \cite[Remark, pp.106-107]{bb}.

Recall that a space $X$ is a
{\em Cantor $n$-manifold} if any partition of $X$ is of dimension $\geq n-1$ \cite{u} (a partition of $X$
is a closed set $P\subset X$ such that $X\setminus P$ is the union of two open nonempty disjoint sets).
In other words, $X$ cannot be the union
of two proper closed sets whose intersection is of covering
dimension $\leq n-2$. {\em Strong Cantor manifolds} is another specification of Cantor manifolds which was considered  by Had\v{z}iivanov \cite{h}.
Had\v{z}iivanov  and Todorov \cite{ht} introduced the class of {\em Mazurkiewicz $n$-manifolds}, which is a proper sub-class of the strong Cantor $n$-manifolds. This notion has its roots in the classical Mazurkiewicz theorem \cite{ma} that
no region $X$ in the Euclidean  $n$-space can be cut by a subset $M$ with $\dim M\leq n-2$ in following sense:  any two points from  $X\setminus M$ can be joined by a continuum $K\subset X\setminus M$.

%$X$ is a Mazurkiewicz $n$-manifold if for any disjoint closed massive subsets $A$ and $B$ of $X$ (a massive subset of $X$ is a set with non-empty %interior in $X$), and any normally placed set $M$ of $X$ with $\dim M\leq n- 2$, there exists a continuum in $X\setminus M$ intersecting $A$ and $B$ %(equivalently, no such $M$ is cutting $X$ between $A$ and $B$). Here, $M$ is a normally placed in $X$ provided every two disjoint closed in $M$ sets %have disjoint open in $X$ neighborhoods (for example, every $F_{\sigma}$ - subset of a normal space is normally placed in that space).

But the strongest specification of Cantor manifolds is the notion of $V^n$-continua introduced by
Alexandroff \cite{ps}: a continuum $X$ is a {\em $V^n$-continuum}
if for every two closed disjoint subsets $X_0$, $X_1$ of $X$, both having non-empty interiors,
there exists an open cover
$\omega$ of $X$ such that there is no partition $P$ in $X$ between
$X_0$ and $X_1$ admitting an $\omega$-map into a space $Y$ with
$\dim Y\leq n-2$ ($f\colon P\to Y$ is said to be an $\omega$-map if there exists
an open cover $\gamma$ of $Y$ such that $f^{-1}(\gamma)$ refines $\omega$).

The above notions are related as follows:
 strong Cantor $n$-manifolds are Cantor $n$-manifolds, every $V^n$ -continuum is a Mazurkiewicz $n$-manifold and Mazurkiewicz $n$-manifolds are strong Cantor $n$-manifolds \cite{ht}. None of the above inclusions is reversible \cite{kktv}.

More
general concepts of the above notions were considered in \cite{kktv} and \cite{tv}. In particular, we are
going to use the following two, where
$\mathcal{C}$ is a class of topological spaces.

\begin{dfn}\label{dfn1}
A connected space $X$ is an {\em Alexandroff manifold with respect to $\mathcal C$}
(br., {\em Alexandroff $\mathcal C$-manifold}) if for every two closed, disjoint subsets $X_0$, $X_1$ of $X$, both having non-empty interiors,
there exists an open cover
$\omega$ of $X$ such that no partition $P$ in $X$ between
$X_0$ and $X_1$ admits an $\omega$-map onto a space $Y\in\mathcal C$.
\end{dfn}

\begin{dfn}\label{dfn2}
A semi-continuum $X$  is said to be a {\em Mazurkiewicz manifold with respect to
$\mathcal{C}$} (br., {\em Mazurkiewicz $\mathcal C$-manifold}) provided for every two closed disjoint
sets $X_0,X_1\subset X$ with non-empty interiors,
and every set $F=\bigcup_{i=0}^\infty F_i\subset X$ with each $F_i\in\mathcal{C}$ being proper closed subset of $X$,
there exists a continuum $K$ in $X\setminus F$ joining $X_0$ and $X_1$. If, in the above definition, $X_0$ and $X_1$ can be joined by an
arc in $X\setminus F$, $X$ is called a {\em Mazurkiewicz arc $\mathcal C$-manifold}.
\end{dfn}

It is still unknown if every compact Alexandroff $D^{n-2}_G$-manifold is a Mazurkiewicz $D^{n-2}_G$-manifold, where $D^{n-2}_G$ is the class of all spaces  whose cohomological dimension $\dim_G$ is $\leq n-2$. Searching for the answer of this question, we consider in Section 2 different types of connectedness between disjoint subsets of compacta. The notions of $K_G^n$-manifolds and strong $K_G^n$-manifolds, introduced in Section 2, play a crucial role in the paper. We also provide in Section 2 some examples of (strong) $K_G^n$-manifolds.

The main result in Section 3 is Theorem 3.4. A particular version of that theorem states that if $(X,F)$ is a strong $K_G^n$-manifold, where $X$ is a metric compactum,  and $M\subset X$ is a set with $H^{n-1}(M,M\cap F;G)=0$, then for any two disjoint nonempty open sets $P$ and $Q$ in $X$ there exists a continuum $K\subset X\setminus M$ joining $P$ and $Q$ provided $M\subset X\setminus (P\cup Q)$. The last requirement can be avoided if $\dim_GM\leq n-1$. One corollary  of this theorem providing a partial answer to the Bing-Borsuk question \cite{bb} was mentioned in the abstract. Another corollary is the following analog of the Mazurkiewicz theorem \cite{ma} cited above: If $M$ is a bounded subset of $\mathbb R^n$ with $\dim M\leq n-1$ and $H^{n-1}(M;\mathbb Z)=0$, then every pair of nonempty disjoint open sets  $P, Q\subset\mathbb R^n$  can be joined by a continuum in $\mathbb R^n\setminus M$.

The starting point for our considerations in Section 4 was the following result of Krupski \cite[Proposition 1.7]{kru93}: {\em Let $X$ be a locally compact locally connected
separable metric space such that for all $k<n$ and $x\in X$ the singular homology groups $H_k(X,X\setminus x)$ are trivial. Then every open connected subset of $U\subset X$ is a Cantor $n$-manifold}. We establish that, under the hypotheses of Krupski's result, every connected open subset of $X$ is a Mazurkiewicz arc manifold with respect to the class of all spaces whose covering dimension is $\leq n-2$ (Theorem 4.1). The same conclusion also holds for  arcwise connected open subsets of a complete metric space which is a product of $n$ metric spaces. We also introduce the $LS^n$-property and show that $LS^{n-2}$ implies $H_k(X,X\setminus x;G)=0$ for every $x\in X$ and $k\leq n-1$ (see Theorem 4.5), where $X$ is a homogeneous locally compact metric $ANR$. This improves a result of Mitchell \cite{mi}.

Everywhere in this paper reduced \v{C}ech cohomology groups are considered. If it is not mentioned otherwise, we suppose that the coefficients of the cohomology groups are from a fixed group $G$  and we omit the symbol for the group of coefficients.

Finally, let us raise the following questions (the second question was answered positively in \cite{ktv} provided $X$ is cyclic in dimension $n$):
\begin{qu}
Let $X$ be a compact Alexandroff manifold with respect to the class $D^{n-2}_G$. Is it true that $X$ is a Mazurkiewicz $D^{n-2}_G$-manifold?
What about if $(X,F)$ is a $K_G^n$-manifold or $X$ is a $V^n_G$-continuum?
\end{qu}

\begin{qu}
Is it true that any homogeneous $ANR$-continuum $X$ of dimension $n$ is an Alexandroff $D^{n-2}_{\mathbb Z}$-manifold?
\end{qu}
{\bf Acknowledgements:} The authors wish to thank the referee for his/her valuable remarks and suggestions which significantly improved the paper.

\section{Alexandroff types connectedness of spaces}

In this section all spaces are assumed to be at least paracompact and $G$ denotes an abelian group. For any space $X$ the cohomological dimension
$\dim_GX$ is the smallest integer $n$ such that $H^{n+1}(X,A;G)=0$ for all closed subsets $A\subset X$ (the reader is referred to \cite{dr} for basic facts on the cohomological dimension of compacta). If $\omega$ is an open cover of $X$ and
$Z\subset X$ is closed, we denote by $|\omega|$ and $|\omega_Z|$ the nerves of $\omega$ and the system $\omega_Z=\{U\cap Z: U\in\omega\}$. For any such $\omega$ and $Z\subset X$ let $p_{\omega_Z}^*\colon H^k(|\omega|,|\omega_Z|)\to H^k(X,Z)$, $k\geq 0$, be the projection between the $k$-th  cohomology groups, where $p_{\omega_Z}\colon (X,Z)\to (|\omega|,|\omega_Z|)$ is a map generated by a partition of unity subordinated to $\omega$ (we call such a map $p_{\omega_Z}$ to be natural). Further, $i_A\colon (A,B)\to (X,Z)$ denotes the embedding of the pair $(A,B)$ into the pair $(X,Z)$.

\begin{dfn} Let $P$ and $Q$ be disjoint nonempty open subsets of  a continuum $X$ and $F\subset X$ closed. We say that the pair $(X,F)$ is {\em $K_G^n$-connected between $P$ and $Q$} if there exist an open cover $\omega$ of $Y=X\setminus (P\cup Q)$ such that the following condition holds for every partition $C$ of $X$ between $P$ and $Q$: any natural map $p_{\omega_{C}}\colon (C,C\cap F)\to (|\omega_C|,|\omega_{C\cap F}|)$ generates a non-trivial homomorphism $$p_{\omega_{C}}^*:H^{n-1}(|\omega_C|,|\omega_{C\cap F}|)\to H^{n-1}(C,C\cap F).$$ If, in the above situation, there exists also $\mathrm{e}\in H^{n-1}(|\omega|,|\omega_{F\cap Y}|)$   such that $p_{\omega_{C}}^*(i^{\ast}_{\omega_{C}}(\mathrm{e}))\neq 0$  for every partition $C$ in $X$ between $P$ and $Q$, the pair $(X,F)$ is called {\em strongly $K^n_G$-connected between $P$ and $Q$}.
Further, $(X,F)$ is said to be a {\em $K^n_G$-manifold} (resp., {\em strong $K^n_G$-manifold}) if it is $K_G^n$-connected (resp., strongly $K^n_G$-connected) between any two nonempty open disjoint sets $P,Q\subset X$. If $F=\varnothing$, we will say that $X$ is a (strong) $K^n_G$-manifold.
\end{dfn}

Definition 2.1 is justified by the following result which was actually established by Kuzminov \cite{ku}:
\begin{thm}
Every compactum $X$ with $\dim_GX=n$ contains a pair $(Y,F)$ of closed sets such that $(Y,F)$ is a strong $K_G^n$-manifold.
\end{thm}
Similarly, the definition of Alexandroff $\mathcal C$-manifold yields the following one:

\begin{dfn}
Let $P$ and $Q$ be open nonempty subsets of  a continuum $X$ with $\overline{P}\cap \overline{Q}=\varnothing$. We say that $X$ is $\mathcal C$-connected
in the sense of Alexandroff between $P$ and $Q$ (shortly, {\em $A(\mathcal C)$-connected between $P$ and $Q$}) if there exists an open cover $\omega$ of
$X\setminus (P\cup Q)$ such that no partition of $X$ between
$P$ and $Q$ admits an $\omega$-map onto a space from the class $\mathcal C$.
\end{dfn}

\begin{pro}
Let $P$ and $Q$ be nonempty open subsets of  a continuum $X$ with $\overline{P}\cap \overline{Q}=\varnothing$ and $F\subset X$ closed. If the pair $(X,F)$ is
$K_G^n$-connected  between $P$ and $Q$, then $X$ is $A(D^{n-2}_G)$-connected between $P$ and $Q$, where $D^{n-2}_G$ is the class of spaces of dimension $\dim_G\leq n-2$.
\end{pro}

\begin{proof}
According to Definition 2.1, we can find an open cover $\omega$ of $Y=X\setminus (P\cup Q)$ satisfying the following condition:
for every partition $C$ in $X$ between $P$ and $Q$ and a natural map $p_{\omega_{C}}\colon (C,C\cap F)\to (|\omega_C|,|\omega_{C\cap F}|)$ there exists an element $\mathrm{e_C}\in H^{n-1}(|\omega_C|,|\omega_{C\cap F}|)$ with $p_{\omega_{C}}^*(\mathrm{e_C})\neq 0$.

Suppose there exists a partition $C$ of $X$ between $P$ and $Q$ admitting an $\omega$-map
$g\colon C\to T$ onto a compactum $T$ with $\dim_GT\leq n-2$. Thus, we can find a finite open cover $\tau$
of $T$  such that
$\nu=g^{-1}(\tau)$ is refining $\omega$. Let $p_\nu\colon (C,C\cap F)\to (|\nu|,|\nu_{C\cap F}|)$ and
$p_\tau\colon (T,g(C\cap F))\to (|\tau|, |\tau_{g(C\cap F)}|)$ be natural maps.
Obviously, the function $V\in\tau\rightarrow g^{-1}(V)\in\nu$ provides
a simplicial homeomorphism $g^{\tau}_\nu\colon (|\tau|, |\tau_{g(C\cap F)}|)\to (|\nu|,|\nu_{C\cap F}|)$. Then the maps $p_\nu$ and
$g_\tau=g^{\tau}_\nu\circ p_\tau\circ g$ are homotopic. Hence, $$p_{\nu}^*=g^*\circ p_{\tau}^*\circ (g^{\tau}_\nu)^*.\leqno{(1)}$$
Because $H^{n-1}(T, g(C\cap F))=0$ (recall that $\dim_GT\leq n-2$), the homomorphism $g^*\colon H^{n-1}(T,g(C\cap F))\to H^{n-1}(C,C\cap F)$
is trivial. Then, by (1), so is the homomorphism $p_{\nu}^*\colon H^{n-1}(|\nu|,|\nu_{C\cap F}|)\to H^{n-1}(C,C\cap F)$.
On the other hand, since $\nu$
refines $\omega$, we can find a map $\varphi_\nu\colon (|\nu|,|\nu_{C\cap F}|)\to (|\omega_C|,|\omega_{C\cap F}|)$ such that $p _{\omega_C}$ and $\varphi_\nu\circ p_{\nu}$
are homotopic. Therefore, $p_{\omega_C}^*=p_{\nu}^*\circ\varphi_\nu^*$.
Since
$p_{\omega_C}^*(\mathrm{e_C})\neq 0$ for some $\mathrm{e_C}\in H^{n-1}(|\omega_C|,|\omega_{C\cap F}|)$,
$p_{\nu}^*(\varphi_\nu^*(\mathrm{e}_C))\neq 0$, which contradicts the triviality of $p_{\nu}^*$.

Therefore, $X$ is $A(D^{n-2}_G)$-connected between $P$ and $Q$.
\end{proof}

\begin{cor}
A compactum $X$ is an Alexandroff $D^{n-2}_G$-manifold provided $(X,F)$ is a $K_G^n$-manifold for some closed set $F\subset X$.
\end{cor}

\begin{dfn}
Let $X$, $P$, $Q$ and $F$ be as in Definition 2.1. We say that the pair $(X,F)$ is {\em $V^n_G$-connected between $P$ and $Q$} if there exists an open cover $\omega$ of $X\setminus (P\cup Q)$ such that $g^*\colon H^{n-1}(T,g(C\cap F))\to H^{n-1}(C,C\cap F)$ is a non-trivial homomorphism for any partition $C$ of $X$ between $P$ and $Q$ and any surjective $\omega$-map $g\colon C\to T$. The pair $(X,F)$ is called a {\em relative $V^n_G$-continuum} provided $(X,F)$ is $V^n_G$-connected between any two open sets $P,Q\subset X$ with disjoint closures.
\end{dfn}

When $F$ is the empty set, the above definition provides $V^n_G$-continua introduced by Stefanov \cite{ss}.
The proof of Proposition 2.4 yields the following stronger conclusion:
\begin{cor}
If a pair $(X,F)$ is $K^n_G$-connected between two open sets $P,Q\subset X$ with disjoint closures, then $(X,F)$ is $V^n_G$-connected between $P$ and $Q$. In particular, every $K^n_G$-manifold $(X,F)$ is a relative $V^n_G$-continuum.
\end{cor}

There is an interesting analogy between
$V^n_G$-continua and relative $V^n_G$-continua. It follows from Theorem 2.2 and Corollary 2.7 that every compactum $X$ with $\dim_GX=n$ contains a relative $V^n_G$-continuum. On the other hand, any compactum $X$ with $H^n(X)\neq 0$ contains a $V^n_G$-continuum (see \cite{ss} for finite-dimensional metric $X$, and \cite{vv} for any compact $X$).

We are going to provide more examples of $K_G^n$-manifolds and strong $K_G^n$-manifolds. A closed non-empty set $A\subset X$ is said to a {\em a cohomological carrier}  of a non-zero element $\alpha\in H^n(X)$ if $i^*_A(\alpha)\neq 0$ and $i^*_B(\alpha)=0$ for every proper closed subset $B\subset A$, where $i_A$ denotes the inclusion map $A\hookrightarrow X$.

Next proposition was established by the second author using different terminology.
\begin{pro}\cite[Proposition 2.5]{vv}
Every cohomological carrier of a non-zero element of $H^n(X)$ is a strong $K_G^n$-manifold.
\end{pro}

Following Yokoi \cite{yo}, a compactum $X$ is called an {\em $(n,G)$-bubble} if $H^n(X)\neq 0$ and $H^n(A)=0$ for every closed proper set $A\subset X$. This is a reformulation of the notion of an {\em $n$-bubble} introduced by Kuperberg \cite{kup} and Choi \cite{ch}, see also Karimov-Repov\v{s} \cite{kr} for the stronger notion of an {\em $H^n$-bubble.} A compactum  $X$ is said to be a {\em generalized $(n,G)$-bubble} \cite{ktv}  if there exists a surjective map $f\colon X\to Y$ such that the homomorphism $f^*\colon H^n(Y)\to H^n(X)$ is nontrivial, but $f_A^*(H^n(Y))=0$ for any proper closed subset $A$ of $X$, where $f_A$ is the restriction of $f$ over $A$.

Proposition 2.9 below was actually established in \cite[Theorem 2.3, Claim 1]{ktv}.
\begin{pro}
Any generalized $(n,G)$-bubble is a strong $K_G^n$-mani-fold.
\end{pro}

\begin{pro}\cite[Theorem 1.1]{vv}
Every homogeneous metric $ANR$-continuum $X$ with $\dim_GX=n$ and $H^n(X)\neq 0$ is an $(n,G)$-bubble, and hence a strong $K_G^n$-manifold.
\end{pro}

\begin{pro}
Let $X$ be a continuum and $F\subset X$ a nonempty nowhere dense closed subset such that the quotient space $X/F$ is a $($strong$)$ $K_G^n$-manifold.
Then $(X,F)$ is also a $($strong$)$ $K_G^n$-manifold.
\end{pro}

\begin{proof}
Suppose $X/F$ is a $K_G^n$-manifold, and $U_1$ and $U_2$ two nonempty disjoint open subsets of $X$. Since $F$ is nowhere dense, there are disjoint open sets $V_1$ and $V_2$ in $X$ such that $\overline{V}_j\subset U_j\setminus F$, $j=1,2$. Then $W_j=\pi(V_j)$, $j=1,2$, are nonempty open and disjoint sets in $X/F$, where $\pi\colon X\to X/F$ is the quotient map. So, there exists a finite open cover $\gamma$ of $(X/F)\setminus (W_1\cup W_2)$ such that for any partition $P$ of $X/F$ between $W_1$ and $W_2$
any natural map $p_{\gamma_{P}}\colon P\to |\gamma_P|$ generates a non-trivial homomorphism $p_{\gamma_{P}}^*:H^{n-1}(|\gamma_P|)\to H^{n-1}(P)$. Hence, $\omega=\pi^{-1}(\gamma)$ is an open cover of $X\setminus (V_1\cup V_2)$. If $C$ is a partition of $X$ between $U_1$ and $U_2$, then $\pi(C)$ is a partition of $X/F$ between $W_1$ and $W_2$, and the function $\gamma\ni T\rightarrow \pi^{-1}(T)\in\omega$ generates a simplicial homeomorphism $\tau\colon (|\gamma_{\pi(C)}|,|\gamma_{\pi(C\cap F)}|)\to (|\omega_C|,|\omega_{C\cap F}|)$. Moreover,
the homomorphism $p_{\gamma_{\pi(C)}}^*:H^{n-1}(|\gamma_{\pi(C)}|)\to H^{n-1}(\pi(C))$ is non-trivial. Since the maps $\tau\circ p_{\gamma_{\pi(C)}}\circ\pi$ and  $p_{\omega_{C}}$ are homotopic (as maps from $(C,C\cap F)$ to $(|\omega_C|,|\omega_{C\cap F}|)$),
$p_{\omega_{C}}^*=\pi^*\circ p_{\gamma_{\pi(C)}}^*\circ\tau^*$. In case $C\cap F=\varnothing$ the complex $|\gamma_{\pi(C\cap F)}|$ is also empty. So, $H^{n-1}(|\gamma_{\pi(C)}|,|\gamma_{\pi(C\cap F)}|)=H^{n-1}(|\gamma_{\pi(C)}|)$. Then the equality $p_{\omega_{C}}^*=\pi^*\circ p_{\gamma_{\pi(C)}}^*\circ\tau^*$ implies that $p_{\omega_{C}}^*$ is non-trivial (recall that $\tau^*$, $\pi^*$ are isomorphisms and $p_{\gamma_{\pi(C)}}^*$ is non-trivial). In case $C\cap F\neq\varnothing$, $\pi(C\cap F)$ is the  point $\pi(F)$ of $X/F$ and $|\gamma_{\pi(C\cap F)}|$ is a simplex of the nerve $|\gamma_{\pi(C)}|$. We have the following commutative diagram, where the vertical maps $p_{\gamma,k}^*$ are the dual maps generated by $p_{\gamma_{\pi(C)}}\colon\pi(C)\to |\gamma_{\pi(C)}|$.
{ $$
\begin{CD}
H^{n-1}(|\gamma_{\pi(C)}|,|\gamma_{\pi(C\cap F)}|)@>{{j_1^*}}>>H^{n-1}(|\gamma_{\pi(C)}|)@>{{i_1^*}}>>H^{n-1}(|\gamma_{\pi(C\cap F)}|)\\
@ VV{p_{\gamma,1}^*}V
@VV{p_{\gamma,2}^*}V@VV{p_{\gamma,3}^*}V\\
H^{n-1}(\pi(C),\pi(C\cap F))@>{{j_2^*}}>>H^{n-1}(\pi(C))@>{{i_2^*}}>>H^{n-1}(\pi(C\cap F))
\end{CD}
$$}\\
Since $|\gamma_{\pi(C\cap F)}|)$ is a simplex, $H^{n-1}(|\gamma_{\pi(C\cap F)}|)=0$. Consequently, $j_1^*$ is surjective, which implies
that $p_{\gamma,1}^*$ is non-trivial because so is $p_{\gamma,2}^*$. Therefore, $(X,F)$ is a $K_G^n$-manifold.

Suppose $X/F$ is a strong $K_G^n$-manifold. Then there exists an element $\mathrm{e}\in H^{n-1}(|\gamma_{\pi(C)}|)$ such that
$p_{\gamma,2}^*(\mathrm{e})\in H^{n-1}(\pi(C))$ is non-trivial. Hence, so is the element $\pi^*(p^*_{\gamma,2}(\mathrm{e}))\in H^{n-1}(C)$ provided
$C\cap F=\varnothing$. But $\pi^*(p^*_{\gamma,2}(\mathrm{e}))=p^*_{\omega_C}((\tau^*)^{-1}(\mathrm{e}))$. Therefore, $(\tau^*)^{-1}(\mathrm{e})$ is an element of  $H^{n-1}(|\omega_C|)$ such that $p^*_{\omega_C}((\tau^*)^{-1}(\mathrm{e}))\in H^{n-1}(C)$ is non-trivial. This shows that $(X,F)$ is a strong $K_G^n$-manifold.
In case $C\cap F\neq\varnothing$, we have $H^{n-1}(|\gamma_{\pi(C\cap F)}|)=0$ because $|\gamma_{\pi(C\cap F)}|$ is a simplex. Then
it follows from the above diagram that there
exists $\mathrm{e}_1\in H^{n-1}(|\gamma_{\pi(C)}|,|\gamma_{\pi(C\cap F)}|)$ with $j_1^*(\mathrm{e}_1)=\mathrm{e}$ and $p_{\gamma,1}^*(\mathrm{e}_1)\neq 0$.
As in the first case, we can show that $p_{\omega_{C}}^*((\tau^*)^{-1}(\mathrm{e_1}))\in H^{n-1}(C,C\cap F)$ is also non-trivial. Hence, $(X,F)$ is a strong $K_G^n$-manifold.
\end{proof}

Let us note that the implication "$X/F$ is a $($strong$)$ $K_G^n$-manifold yields $(X,F)$ is a $($strong$)$ $K_G^n$-manifold" in Proposition 2.11 can not be inverted. For example, $(\mathbb S^1,F)$ is a strong $K_{\mathbb Z}^1$-manifold but $\mathbb S^1/F$, where $F$ is a two-point subset of $\mathbb S^1$, is not a $K_{\mathbb Z}^1$-manifold because the identification point $\pi (F)$ is a partition of $\mathbb S^1/F$.
\begin{cor}
The pair $(\mathbb I^n,\mathbb S^{n-1})$, $n\geq 1$, is a strong $K_{\mathbb Z}^n$-manifold.
\end{cor}

\begin{proof}
This follows from Proposition 2.11 because $\mathbb I^n/\mathbb S^{n-1}=\mathbb S^n$ is a strong $K_{\mathbb Z}^n$-manifold according to Proposition 2.10.
\end{proof}

The next proposition provides more examples of strong $K_G^n$-manifolds (recall that the Eilenberg-MacLane complexes $K(G,n)$ have the following property: $\dim_GX\leq n$ if and only if any map $g\colon A\to K(G,n)$ can be extended over $X$, where $X$ is compact and $A\subset X$ is closed).
\begin{pro}
Let $X$ be a continuum  and $F\subset X$ a closed nowhere dense subset of $X$. If there exists a map $f\colon F\to K(G,n-1)$, which is not extendable over $X$ but it is extendable over $Y\cup F$ for any proper closed subset $Y$ of $X$, then $(X,F)$ is a strong $K_G^n$-manifold. In particular, if $\dim X=n$ and
$f\colon F\to\mathbb S^{n-1}$ is a map not extendable over $X$ but extendable over any proper closed subset of $X$ containing $F$, then $(X,F)$ is a strong
$K_{\mathbb Z}^n$-manifold.
\end{pro}

\begin{proof}
Since $H^{n-1}(F)$ is isomorphic to the group $[F;K(G,n-1)]$ of pointed homotopy classes of maps from $F$ to $K(G,n-1)$, let $\alpha\in H^{n-1}(F)$ be the homotopy class $[f]$ of $f$. Let $A$ be a proper closed subset of $X$. Consider the following diagram, where the vertical arrows are homomorphisms generated by the corresponding inclusions
{ $$
\begin{CD}
H^{n-1}(X)@>{{i_1^*}}>>H^{n-1}(F)@>{{\delta_1}}>>H^{n}(X,F)\\
@ VV{j_1^*}V
@VV{j_2^*}V@VV{j_3^*}V\\
H^{n-1}(A)@>{{i_2^*}}>>H^{n-1}(A\cap F)@>{{\delta_2}}>>H^{n}(A,A\cap F).
\end{CD}
$$}\\
Because $f$ is not extendable over $X$, $\alpha\not\in i_1^*(H^{n-1}(X))$. Hence, $\beta=\delta_1(\alpha)$ is a non-trivial element of
$H^{n}(X,F)$. On the other hand, since the interior of $F$ in $X$ is empty, $F\cup A\neq X$. So,
$f$ is extendable over $A\cup F$, which implies that the restriction $f|(A\cap F)$ is extendable over $A$.
Hence, $[f|(A\cap F)]=j_2^*(\alpha)$ belongs to $i_2^*(H^{n-1}(A))$. So, $\delta_2(j_2^*(\alpha))=j_3^*(\beta)=0$. In this way we proved that $\beta$ is a non-zero element of $H^{n}(X,F)$ whose image under the homomorphism $j_3^*\colon H^{n}(X,F)\to H^{n}(A,A\cap F)$ is trivial for every proper closed subset $A\subset X$.

Now, following the proof of Theorem 1 from \cite{ku}, one can show that $(X,F)$ is a strong $K_G^n$-manifold.
Suppose $U_1$ and $U_2$ are non-empty open subsets of $X$ with disjoint closures, and let $m_k:Y_k\hookrightarrow X$ be the inclusion of
$Y_k=X\backslash U_k$ into $X$, $k=1,2$, and $Y=Y_1\cap Y_2$. Consider the Mayer-Vietoris exact sequence
{ $$
\begin{CD}
H^{n-1}(Y,Y\cap F)@>{{\delta}}>>H^{n}(X,F)@>{{j}}>>\bigoplus\limits_{k=1}\limits^{2}H^{n}(Y_k,Y_k\cap F)
\end{CD}
$$}\\
with $j=(m_1^*,m_2^*)$. Since $j(\beta)=0$, there exists a non-trivial element $\gamma\in H^{n-1}(Y,Y\cap F)$ with $\delta(\gamma)=\beta$.
Consequently, there exist
an open cover $\omega$ of $Y$  and
$\mathrm{e}\in H^{n-1}(|\omega|,|\omega_{Y\cap F}|)$ with $p_\omega^*(\mathrm{e})=\gamma$.

Let $C$ be a partition of $X$ between $U_1$ and $U_2$. So, $X=P_1\cup P_2$ and $C=P_1\cap P_2$, where each $P_k$ is
a closed subset of $X$ containing $U_k$, $k=1,2$. Denote by $i:C\hookrightarrow Y$, $i_{12}:P_1\hookrightarrow Y_2$
and $i_{21}:P_2\hookrightarrow Y_1$ the corresponding inclusions.
Then we have the following commutative diagram, whose rows are Mayer-Vietoris
sequences:

{ $$
\begin{CD}
H^{n-1}(Y,Y\cap F)@>{{\delta}}>>H^{n}(X,F)@>{{j}}>>\bigoplus\limits_{k=1}\limits^{2}H^{n}(Y_k,Y_k\cap F)\\
@ VV{i^*}V
@VV{id}V@VV{i_{12}^*\oplus i_{21}^*}V\\
H^{n-1}(C,C\cap F)@>{{\delta_C}}>>H^{n}(X,F)@>{{j_C}}>>\bigoplus\limits_{k=1}\limits^{2}H^{n}(P_k,P_k\cap F).
\end{CD}
$$}\\
Obviously, $$\delta_C(i^*(\gamma))=id(\delta(\gamma))=\beta\neq 0.\leqno{(2)}$$ On the other hand, the commutativity of the diagram

{ $$
\begin{CD}
H^{n-1}(|\omega|,|\omega_{Y\cap F}|)@>{{p_{\omega}^*}}>>H^{n-1}(Y,Y\cap F)\\
@ VV{i_C^*}V
@VV{i^*}V\\
H^{n-1}(|\omega_C|,|\omega_{C\cap F}|)@>{{p_{\omega_C}^*}}>>H^{n-1}(C,C\cap F)
\end{CD}
$$}\\
implies that $p_{\omega(C)}^*(i_C^*(\mathrm{e}))=i^*(p_{\omega}^*(\mathrm{e}))=i^*(\gamma)$. Therefore, according to (2),
$p_{\omega(C)}^*(i_C^*(\mathrm{e}))\neq 0$.

Suppose now that $\dim X=n$ and $f\colon F\to\mathbb S^{n-1}$ is a map not extendable over $X$ but extendable over any proper closed subset of $X$ containing $F$. Since $\mathbb S^{n-1}$ as an $n$-skeleton of $K(\mathbb Z,n-1)$, $f$ considered as a map to  $K(\mathbb Z,n-1)$ is not extendable over $X$ because any such an extension would yield an extension of $f$ (see the proof of \cite[Theorem 1.4]{dr}). So, $(X,F)$ is a strong $K_{\mathbb Z}^n$-manifold.
\end{proof}

We say that an $n$-system $\mathrm{S}=\{(F_i^+,F_i^-):i=1,..,n\}$ of pairs with any pair consisting of closed disjoint subset of $X$ is {\em essential} if for any sequence $\{C_i\}$ of partitions $C_i$ of $X$ between $F_i^+$ and $F_i^-$ we have $\bigcap_{i=1}^{n}C_i\neq\varnothing$.
A space $X$ is said to be {\em spanned on $\mathrm{S}$} if $\mathrm{S}$ is essential but the system $\mathrm{S_Y}=\{(F_i^+\cap Y,F_i^-\cap Y):i=1,..,n\}$ is not essential for every proper closed subset $Y\subset X$. Observe that if $S$ spans $X$, then the set $F=\bigcup_{i=1}^{n}(F_i^+\cup F_i^-)$ is nowhere dense in $X$.

\begin{cor}
Let $X$ be an $n$-dimensional continuum and $\mathrm{S}$ be an $n$-system in $X$, $n\geq 2$, spanning $X$.  Then $(X,F)$ is a strong $K_{\mathbb Z}^n$-manifold, where $F=\bigcup_{i=1}^{n}(F_i^+\cup F_i^-)$.
\end{cor}

\begin{proof}
Let $g_i\colon X\to J$, $J=[-1,1]$, be a function with $g_i(F_i^+)=1$ and $g_i(F_i^-)=-1$, $i=1,..,n$. Denote by $g\colon X\to\mathbb J^n$ the diagonal product   map $g(x)=(g_1(x),...,g_n(x))$, $x\in X$. Obviously, $f=g|F$ is a map from $F$ into $\mathbb S^{n-1}$ (here $\mathbb S^{n-1}$ is identified with the boundary of $\mathbb J^n$).

Since $\mathrm{S}$ is essential on $X$ and inessential on every proper closed subset of $X$, the map $f$ is not extendable over $X$ but it is extendable over any proper closed set of $X$ containing $F$.
Indeed, suppose $\bar{f}\colon X\to\mathbb S^{n-1}$ is an extension of $f$. Then for each $i$ the set $C_i=\bar{f}^{-1}(P_i)$, where $P_i=\{x\in\mathbb I^n:x_i=0\}$,  is a partition of $X$ between $F_i^+$ and $F_i^-$ and $\bigcap_{i=1}^{n}C_i=\varnothing$, a contradiction. If $Y$ is a proper closed subset of $X$ containing $F$, there exist partitions $L_i$ of $X$ between $F_i^+$ and $F_i^-$ such that $\bigcap_{i=1}^{n}L_i\cap Y=\varnothing$.
According to \cite[Lemma 1, pp. 339]{ap}, there are closed $G_\delta$-sets $D_i$ in $Y$, which are partitions between $F_i^+$ and $F_i^-$, such that such that $\bigcap_{i=1}^{n}D_i=\varnothing$
and $Y\cap L_i\subset D_i$ for each $i$. Then we can construct functions
$h_i\colon Y\to [-1,1]$ with $h_i(F_i^+)=1$, $h_i(F_i^-)=-1$ and $D_i=h_i^{-1}(0)$, $1\leq i\leq n$. Obviously, the diagonal product $h$ of all $h_i$ is a map from $Y$ into $\mathbb J^n$ such that $a=(0,...,0)\not\in h(Y)$. Let $r\colon\mathbb J^n\setminus\{a\}\to\mathbb S^{n-1}$ be a retraction and $\widetilde{f}=r\circ h$. It follows from the definition of $f$ and $\widetilde{f}$ that $f(x)$ and $\widetilde{f}(x)$ belong to the same face of $\mathbb J^n$ for every $x\in F$. This implies that $f$ and $\widetilde{f}|F$ are homotopic. So, by the homotopy extension theorem, $f$ is extendable over $Y$.

Finally, according to Proposition 2.13, $(X,F)$ is a strong $K_{\mathbb Z}^n$-manifold.
\end{proof}

Corollary 2.14 extends a result proved independently by Hamamdzhi-ev \cite{ham} and the first author \cite{vt1} that, under the hypotheses of this corollary, $X$ is a $V^n$-continuum.

\section{$K^n_G$-manifolds and Mazurkiewicz manifolds}
Another type of connectedness is inspired by Mazurkiewicz $\mathcal C$-mani-folds, where $\mathcal C$ is a given class of spaces.

\begin{dfn}
Let $P, Q\subset X$ be nonempty open subsets of $X$ with disjoint closures. The space $X$ is said be {\em $M(\mathcal C)$-connected between $P$ and $Q$} if
for every set $F\subset X$ with $F\in\mathcal C$ there exists a continuum $K\subset X\setminus F$ connecting $P$ and $Q$.
\end{dfn}

It is well known \cite{ht} that any continuum $X$ which is $A(D^{n-2})$-connected between two nonempty open sets $P, Q\subset X$ with disjoint closures
 is $M(\mathcal C)$-connected between $P$ and $Q$ with respect to the class $\mathcal C$ of all normally placed $(n-2)$-dimensional subsets of $X$ (here $D^{n-2}$ is the class of all spaces of covering dimension $\dim\leq n-2$). We don't know if this holds when the covering dimension is replaced by the cohomological dimension $\dim_G$. But this is true if instead of $A(D^{n-2})$-connectedness the stronger $K^n_G$-connectedness of $X$ is assumed (recall that a set $M$ is a normally placed in $X$ provided every two disjoint closed in $M$ sets have disjoint open in $X$ neighborhoods; for example, every $F_{\sigma}$ - subset of a normal space is normally placed in that space).

\begin{thm}  Let the pair $(X,F)$ be strongly $K^n_G$-connected between two nonempty open disjoint sets $P, Q\subset X$. Then for every
Lindel\"{o}ff normally placed subset $M\subset X$ with $\dim_G M\leq n-2$ there exists a continuum $K\subset X\setminus M$  connecting $P$ and $Q$.
\end{thm}

\begin{proof}
Suppose there exists a Lindel\"{o}ff normally placed subset $M\subset X$ with $\dim_G M\leq n-2$ such that any continuum $K$ connecting $P$ and $Q$ meets $M$. According to \cite[Lemma 2.5]{kktv}, there are nonempty open sets $P_1, Q_1\subset X$ such that $\overline{P}_1\subset P$,
$\overline{Q}_1\subset Q$ and if $K$ is a continuum in $X$ joining $\overline{P}_1$ and $\overline{Q}_1$, then $K\cap M\setminus(P_1\cup Q_1)\neq\varnothing$. We are going to show this is not true, which will complete the proof.

Let $\omega$ be a finite open cover of $Y=X\setminus (P_1\cup Q_1)$ and $\mathrm{e}$ an element of $H^{n-1}(|\omega|,|\omega_{F\cap Y}|)$, both satisfying the requirements from Definition 2.1 with $P$ and $Q$ replaced by $P_1$ and $Q_1$, respectively.
Obviously, $M_1=M\setminus(P_1\cup Q_1)$ is closed in $M$. So, $M_1$ is also Lindel\"{o}ff and normally placed in $Y$. Moreover,
 $\dim_G M_1\leq n-2$, which implies $H^{n-1}(M_1,M_1\cap F)=0$. Consequently, there exist an open cover $\gamma$ of $M_1$ refining $\omega$ and a simplicial map $$\pi^\gamma_\omega:(|\gamma|,|\gamma_{M_1\cap F}|)\rightarrow (|\omega_{M_1}|,|\omega_{M_1\cap F}|)\hbox{~}\mbox{such that}\hbox{~}
(\pi^\gamma_\omega)^*(i_{M_1}^*(\mathrm{e}))=0,$$
where $i_{M_1}$ is the embedding $(|\omega_{M_1}|,|\omega_{M_1\cap F}|)\hookrightarrow (|\omega|,|\omega_{F}|)$.
Since $M_1$ is Lindel\"{o}ff, we may assume that $\gamma$ is countable. Next, let $\gamma_1=\{A_i:1,2,..\}$ be a system of closed in $M_1$ sets covering $M_1\cap F$ such that $A_i\subset U_i\cap F$, where $U_i\in\gamma$, for all $i$.
Using the proof of \cite[Theorem 3.1.1]{re}, we find an open in $M_1$ system  $\{L_i:i=1,2,..\}$ such that $\lambda_1=\{\overline{L}^{M_1}_i:i=1,2,..\}$ is a swelling of $\gamma_1$ with $\overline{L}^{M_1}_i\subset U_i$, $i\geq 1$, where $\overline{L}^{M_1}_i$ is the closure of $L_i$ in $M_1$. Recall that $\lambda_1$ is a swelling of $\gamma_1$ if each $\overline{L}^{M_1}_i$ contains $A_i$ and
for any finite set of indices $i_1,..,i_s$ we have $\bigcap_{k=1}^{s}\overline{L}^{M_1}_{i_k}\neq\varnothing$ iff $\bigcap_{k=1}^{s}A_{i_k}\neq\varnothing$.
Consider the Lindel\"{o}ff space $M_1\setminus L$, where $L=\bigcup_{i\geq 1}L_i$. If $M_1\setminus L$ is non-empty, there exists a countable system $\lambda_2=\{B_i:i=1,2,..\}$ of closed subsets of $M_1$ such that $\lambda_2$ covers $M_1\setminus L$, $\lambda_2$ refines $\gamma$ and all $B_i$ are contained in $M_1\setminus F$. In case $M_1\setminus L=\varnothing$, $\lambda_2$ is the empty family. Obviously, $\lambda=\lambda_1\cup\lambda_2$ is a closed countable cover of $M_1$ refining $\gamma$. Let $\lambda=\{\Lambda_i:i=1,2,..\}$. Using again the proof of \cite[Theorem 3.1.1]{re}, we find an open in $M_1$ swelling $\theta=\{W_i:i=1,2,..\}$ of $\lambda$ such that $\theta$ refines $\gamma$. Moreover, we can insist that $W_i\cap F=\varnothing$ provided $\Lambda_i\cap F=\varnothing$.

\textit{Claim. There exists an open in $Y$ swelling $\tilde{\sigma}=\{V_i:i=1,2,..\}$  of $\lambda$ refining $\omega$ such that $\Lambda_i\subset\overline{V}_i\cap M_1\subset W_i$ for all $i$.}

We follow the proof of Lemma 4.1.2 from \cite{vt1}. Obviously, it suffices to construct by induction open in $Y$ swellings $\sigma_m=\{V_1,...,V_m\}$ of $\{\Lambda_1,...,\Lambda_m\}$, $m\geq 1$, such that
\begin{itemize}
\item $\sigma_{m+1}=\sigma_m\cup\{V_{m+1}\}$;
\item $\Lambda_k\subset V_k\cap M_1\subset\overline{V}_k\cap M_1\subset W_k$, $k\geq 1$;
\item each $V_i$ is contained in an element of $\omega$.
\end{itemize}
 Suppose we already constructed $\sigma_m$ for some $m$. Since $\Lambda_{m+1}$ and $M\setminus W_{m+1}$ are closed disjoint subsets of $M_1$ and $M_1$ is normally placed in $Y$, there are disjoint open in $Y$ sets $V_{m+1}'$ and $\tilde{V}_{m+1}$ with $\Lambda_{m+1}\subset V_{m+1}'$ and $M_1\setminus W_{m+1}\subset\tilde{V}_{m+1}$. Let $E$ be the union of all intersections $\overline{V}_{i_1}\cap...\cap\overline{V}_{i_s}$, $1\leq i_k\leq m$, such that $\bigcap_{k=1}^{s}\Lambda_{i_k}\cap\Lambda_{m+1}=\varnothing$. We claim that $E\cap\Lambda_{m+1}=\varnothing$. Indeed, otherwise  $\bigcap_{k=1}^{p}\overline{V}_{j_k}\cap\Lambda_{m+1}\neq\varnothing$ for some indices $j_1,..,j_p$ such that $j_k\leq m$ and $\bigcap_{k=1}^{p}\Lambda_{j_k}\cap\Lambda_{m+1}=\varnothing$. Then
 $$\bigcap_{k=1}^{p}\overline{V}_{j_k}\cap\Lambda_{m+1}\subset\bigcap_{k=1}^{p}\overline{V}_{j_k}\cap M_1\cap \Lambda_{m+1}\subset\bigcap_{k=1}^{p}W_{j_k}\cap W_{m+1}\neq\varnothing.$$ Since $\{W_i:i=1,2,..\}$ is a swelling of $\lambda$, this implies that $\bigcap_{k=1}^{p}\Lambda_{j_k}\cap\Lambda_{m+1}\neq\varnothing$, a contradiction. To finish the inductive step, observe that $\lambda$ refines $\omega$. So, there exists $O\in\omega$ containing $\Lambda_{m+1}$. We set $V_{m+1}=V_{m+1}'
 \cap O\setminus E$,  and
$\sigma_{m+1}=\sigma_m\cup\{V_{m+1}\}$. It is easily seen that $\sigma_{m+1}$ is a swelling of $\{\Lambda_i:i=1,2,..,m+1\}$, which completes the proof of the claim.

Recall that $\theta$ is a refining of $\gamma$. Moreover, any refining function between $\theta$ and $\gamma$ provides a simplicial map
$$\pi^\theta_\gamma:(|\theta|,|\theta_{M_1\cap F}|)\rightarrow (|\gamma_{M_1}|,|\gamma_{M_1\cap F}|)\hbox{~}\mbox{such that} (\pi^\theta_\omega)^*(i_{M_1}^*(\mathrm{e}))=0,$$
where $\pi^\theta_\omega$ is the composition $\pi^\gamma_\omega\circ\pi^\theta_\gamma$. Let $\theta_j=\{W_i:\Lambda_i\in\lambda_j\}$, $j=1,2$. Obviously, $\theta=\theta_1\cup\theta_2$.
Consider the open in $Y$ system
$$\sigma=\{V_i\setminus F:W_i\in\theta_2\}\cup\{V_i:W_i\in\theta_1\}\hbox{~}\mbox{and let}\hbox{~}W=\bigcup_{\Sigma\in\sigma}\Sigma.$$ It follows from our construction that $\sigma=\{\Sigma_i:i=1,2,..\}$ has the following property: for any indices $i_1,..,i_s$ we have
$$\bigcap_{k=1}^{s}\Sigma_{i_k}\neq\varnothing\hbox{~}\mbox{iff}\hbox{~}\bigcap_{k=1}^{s}W_{i_k}\neq\varnothing\hbox{~}\mbox{and}\hbox{~}
\bigcap_{k=1}^{s}\Sigma_{i_k}\cap F\neq\varnothing\hbox{~}\mbox{iff}\hbox{~}\bigcap_{k=1}^{s}W_{i_k}\cap F\neq\varnothing.$$
 Therefore, there exists a simplicial homeomorphism $\xi$ between the nerves $(|\theta|,|\theta_{M_1\cap F}|)$ and $(|\sigma|,|\sigma_{W\cap F}|)$. So, $\xi^{-1}$ is also a simplicial homeomorphism from  $(|\sigma|,|\sigma_{W\cap F}|)$ onto $(|\theta|,|\theta_{M_1\cap F}|)$. Then for the simplicial map $\pi$ below we have
$$\pi=i_{M_1}\circ\pi_\omega^\theta\circ\xi^{-1}\colon (|\sigma|,|\sigma_{W\cap F}|)\to (|\omega|,|\omega_{F\cap Y}|), \hbox{~}\pi^*(\mathrm{e})=0.$$

Suppose the set $W$ contains a partition $C$ of $X$ between $\overline{P}_1$ and $\overline{Q}_1$ and consider the following diagram, where $i$ and $i_C$ are natural inclusions:
\[
\begin{diagram}
\node{(|\sigma_C|,|\sigma_{F\cap C}|)}\arrow{e,t}{i}\node{(|\sigma|,|\sigma_{W\cap F}|)}\arrow{e,t}{\pi}\node{(|\omega|,|\omega_{F\cap Y}|)}\\
\node{(C,C\cap F)}\arrow{n,r}{p_{\sigma_C}}\arrow[2]{e,t}{i_C}\node[2]{(Y,F\cap Y)}\arrow{n,r}{p_{\omega}}\\
\end{diagram}
\]  It follows from our construction that for every $x\in C$ (resp., $x\in C\cap F$) both points $\pi(i(p_{\sigma_C}(x)))$ and  $p_\omega(i_C(x))$ belong to a simplex of the nerve $|\omega|$ (resp. $|\omega_{F\cap Y}|$). So, the maps $\pi\circ i\circ p_{\sigma_C}$ and $p_\omega\circ i_C$ are homotopic and we have the commutative diagram:

\[
\begin{diagram}
\node{H^{n-1}(|\omega|,|\omega_{F\cap Y}|)}\arrow{e,t}{\pi^*}\arrow{s,r}{p_{\omega}^*}
\node{H^{n-1}(|\sigma|,|\sigma_{W\cap F}|)}\arrow{e,t}{i^*}
\node{H^{n-1}(|\sigma_C|,|\sigma_{C\cap F}|)}\arrow{s,r}{p_{\sigma_C}^*}\\
\node{H^{n-1}(Y,F\cap Y)}\arrow[2]{e,t}{i_C^*}\node[2]{H^{n-1}(C,C\cap F)}
\end{diagram}
\]
Because $\pi^*(\mathrm{e})=0$, the last diagram implies
$$i_C^*(p_\omega^*(\mathrm{e})=0.\leqno{(2)}$$
On the other hand, the following diagram

\[
\begin{diagram}
\node{H^{n-1}(|\omega|,|\omega_{F\cap Y}|)}\arrow{e,t}{i_{\omega_C}^*}\arrow{s,r}{p_{\omega}^*}%\node{H^{n-1}(|\omega_C|,|\omega_{C\cap F}|)}
\node{H^{n-1}(|\omega_C|,|\omega_{C\cap F}|)}\arrow{s,r}{p_{\omega_C}^*}\\
\node{H^{n-1}(Y,F\cap Y)}\arrow{e,t}{i_C^*}\node{H^{n-1}(C,C\cap F)}
\end{diagram}
\]
yields, by $(2)$,
$i_C^*(p_\omega^*(\mathrm{e})=p_{\omega_{C}}^*(i^{\ast}_{\omega_{C}}(\mathrm{e}))=0$, a contradiction.

Hence, $W$ does not contain any partition of $X$ between $\overline{P}_1$ and $\overline{Q}_1$. Choose an open set $\tilde{W}$ such that both sets $\overline{P}_1\setminus\tilde{W}$ and $\overline{Q}_1\setminus\tilde{W}$ are non-empty and $\tilde{W}\cap Y=W$.
%Indeed, otherwise there would be an open set $V\in X$ with $\overline{P}_1\subset V\subset\overline{V}\subset W\setminus\overline{Q}_1$. This would yield that the boundary of $\overline{V}$ is contained in $W$ and separates $X$ between $\overline{P}_1$ and $\overline{Q}_1$. Similarly, we have $\overline{Q}_1\setminus W\neq\varnothing$.
Suppose that $X\setminus\tilde{W}$ is not connected between $\overline{P}_1\setminus\tilde{W}$ and $\overline{Q}_1\setminus\tilde{W}$, i.e. $X\setminus\tilde{W}=A\cup B$ with $A, B$ being disjoint closed sets such that $\overline{P}_1\setminus\tilde{W}\subset A$ and $\overline{Q}_1\setminus\tilde{W}\subset B$. Then the sets $B'=B\cup\overline{Q}_1$
and $A'=A\cup\overline{P}_1$ are disjoint and closed in $X$ whose union is a proper subset of $X$. So, there is a partition $C$ in $X$ between $A'$ and $B'$. Obviously, $C$ is a partition between $\overline{P}_1$ and $\overline{Q}_1$ which is contained in $W$, a contradiction. Hence,
$X\setminus\tilde{W}$ is connected between
$\overline{P}_1\setminus\tilde{W}$ and $\overline{Q}_1\setminus\tilde{W}$.
This implies (see \cite[\S47, Theorem 3, p. 170]{kur}) the existence of a continuum $K\subset X\setminus\tilde{W}$ connecting $\overline{P}_1$ and $\overline{Q}_1$. Finally, since $M_1\subset\tilde{W}$, $K\subset X\setminus M_1$. This contradicts the fact that any continuum connecting $\overline{P}_1$ and $\overline{Q}_1$ should meet $M_1$.
\end{proof}

\begin{cor}
Every strong $K^n_G$-manifold is a Mazurkiewicz $D^{n-2}_G$-manifold.
\end{cor}

The proof of Theorem 3.2 provides a stronger conclusion.

\begin{thm}  Let the pair $(X,F)$ be strongly $K^n_G$-connected between two nonempty open disjoint sets $P, Q\subset X$ and $M$
be a Lindel\"{o}ff normally placed subset of $X$ with $H^{n-1}(M,M\cap F)=0$.
Then in each of the following two cases there exists a continuum $K\subset X\setminus M$  connecting $P$ and $Q$.
\begin{itemize}
\item[(1)] $M\subset X\setminus (P\cup Q)$;
\item[(2)] $\dim_GM\leq n-1$ and $F\cap M$ is a $G_\delta$-set in $M$.
\end{itemize}
\end{thm}

\begin{proof}
It suffices to prove the theorem in case $\overline{P}\cap\overline{Q}=\varnothing$. The proof of item (1) follows from the proof of Theorem 3.2 considering $M$, $P$ and $Q$ instead of $M_1$, $P_1$ and $Q_1$, respectively. For the second item, according to the proof of Theorem 3.2, we need to show that $H^{n-1}(M_1,M_1\cap F)=0$, where $M_1=M\setminus (P\cup Q)$. That will be done if
we show that the homomorphism $H^{n-1}(M,M\cap F)\to H^{n-1}(M_1, F\cap M_1)$ is surjective. To this end, consider the quotient map $\pi\colon M\to M/(M\cap F)$ (here $M/(M\cap F)$ is the quotient space obtained from $M$ by shrinking $M\cap F$ to a point). Since $F\cap M$ is a $G_\delta$ set in $M$, $M\setminus F$ is an $F_\sigma$-subset of $M$. This implies that $\dim_G M\setminus F\leq n-1$. On the other hand,  $(M/(M\cap F))\setminus\{\pi(M\cap F)\}$ is homeomorphic to $M\setminus F$, so, $\dim_GM/(M\cap F)\leq n-1$.
Consequently, the homomorphism $H^{n-1}(M/(M\cap F))\to H^{n-1}(M_1/(F\cap M_1))$ is surjective.
Finally, by a result of Bartik \cite{ba}, $H^{n-1}(M,M\cap F)=H^{n-1}(M/F)$ and
$H^{n-1}(M_1, M_1\cap F)=H^{n-1}(M_1/(F\cap M_1))$, which completes the proof.
\end{proof}

%\begin{cor}
%Let $(X,F)$ $($resp., $X$$)$ be a strong $K^n_G$-manifold and $M$ be a Lindel\"{o}ff normally placed subset of $X$ such that $H^{n-1}(M,M\cap F)=0$ %$($resp., $H^{n-1}(M)=0$$)$. Then for every two disjoint open sets $P,Q$ in $X$, which are contained in $X\setminus M$, there exists a %continuum $K\subset X\setminus M$ joining $P$ and $Q$.
%\end{cor}

We say that a subset $M\subset X$  cuts $X$ between two disjoint sets $A,B\subset X$ if $(A\cup B)\cap M=\varnothing$ and every continuum in $X$ joining $A$ and $B$ meets $M$. The next corollary follows from Theorem 3.4 and Proposition 2.10.

\begin{cor}
Let $X$ be a homogeneous metric $ANR$-continuum with $\dim_GX=n$ and $H^{n}(X)\neq 0$. Then $H^{n-1}(M)\neq 0$ for every set $M\subset X$, which cuts $X$ between two disjoint open subsets of $X$.
\end{cor}
Corollary 3.5 is interesting because the Bing-Borsuk question \cite{bb} whether $H^{n-1}(M)\neq 0$ for any partition of a homogeneous metric $ANR$-space $X$ of dimension $n$ is still unanswered.
When $M$ is a partition of $X$, this corollary was established by the second author in \cite{vv}.

%\begin{cor}
%Let  $(X,F)$ be a strong $K^n_G$-manifold, where $F$ is a closed $G_\delta$-set in $X$, and $M\subset X$ a compactum with $\dim_GM\leq n-1$ and
%$H^{n-1}(M,M\cap F;G)=0$. Then every disjoint open sets $P, Q\subset X$ can be joined with a continuum $K$ such that $K\cap M\setminus (P\cup %Q)=\varnothing$. The same conclusion holds if $F=\varnothing$ and $M$ is a Lindel\"{o}ff normally placed subset of $X$.
%\end{cor}

%\begin{proof}
%According to Theorem 3.4, it suffices to show that the homomorphism $H^{n-1}(M,M\cap F)\to H^{n-1}(M_1, F\cap M_1)$ is surjective, where %$M_1=M\setminus (P\cup Q)$. To this end, consider the quotient map $\pi\colon X\to X/F$ (here $X/F$ is the quotient space obtained from $M$ by %shrinking $M\cap F$ to a point). Since $F$ is a $G_\delta$ set in $X$, $M\setminus F$ is an $F_\sigma$-subset of $M$. This implies that $\dim_G %M\setminus F\leq n-1$. On the other hand,  $(M/F)\setminus\{\pi(F)\}$ is homeomorphic to $M\setminus F$, so, $\dim_GM/F\leq n-1$.
%Similarly, $\dim_G(M_1/(F\cap M_1)\leq n-1$. Consequently, the homomorphism $H^{n-1}(M/F)\to H^{n-1}(M_1/(F\cap M_1))$ is %surjective.
%Finally, the identifications $H^{n-1}(M,M\cap F)=H^{n-1}(M/F)$ and
%$H^{n-1}(M_1, M_1\cap F)=H^{n-1}(M_1/(F\cap M_1))$ complete the proof.

%In the case $F=\varnothing$ and $M$ is a Lindel\"{o}ff normally placed subset of $X$, we directly conclude that the homomorphism %$i^*\colonH^{n-1}(M)\toH^{n-1}(M\setminus (P\cup Q))$ is surjective.
%\end{proof}

The next corollary can be compared with the classical Mazurkiewicz theorem \cite{ma} that
 any region $X$ in the Euclidean  space $\mathbb R^n$ has the following property: if $M\subset X$ with $\dim M\leq n-2$, then  every two points from  $X\setminus M$ can be joined by a continuum $K\subset X\setminus M$.
\begin{cor}
Let $M$ be either a subset of $\mathbb S^n$ or a bounded subset of $\mathbb R^n$ with $H^{n-1}(M;\mathbb Z)=0$. Then every pair of disjoint open sets  $P, Q\subset\mathbb S^n$ $($resp., $P, Q\subset\mathbb R^n$$)$ such that $(P\cup Q)\cap M=\varnothing$ can be joined by a continuum in $\mathbb S^n\setminus M$ $($resp., in $\mathbb R^n\setminus M$$)$. If, in addition $\dim M\leq n-1$, the requirement $(P\cup Q)\cap M=\varnothing$ can be removed.
\end{cor}

\begin{proof}
Obviously, the case when $M\subset\mathbb S^n$ follows from Corollary 3.5. When $M$ is a bounded subset of $\mathbb R^n$, we take an $n$-dimensional cube $B$ whose interior contains $M$ and meets both $P$ and $Q$. Then the proof follows from Corollary 2.12 and Theorem 3.4.
\end{proof}

Let us mention that the requirement in Corollary 3.6 $M$ to be bounded in $\mathbb R^n$ is essential. Indeed, any $(n-1)$-dimensional hyperplane in $\mathbb R^n$ is a counterexample. Also, because every $n$-dimensional set in $\mathbb R^n$ has a non-empty interior, the condition $(P\cup Q)\cap M=\varnothing$ can not be dropped unless $\dim M\leq n-1$.

\section{Homology manifolds and Mazurkiewicz manifolds}

In this section we are going to show that some homological properties of a metric space $X$ imply that $X$ is a {\em Mazurkiewicz arc $n$-manifold} in the following sense: If $M$ is an $F_\sigma$-subset of $X$ with $\dim M\leq n-2$, then any two disjoint sets $A$ and $B$, both having non-empty interiors, can be joined by an arc. Obviously, every Mazurkiewicz arc $n$-manifold is a Mazurkiewicz  manifold with respect to the class of all spaces whose covering dimension $\dim$  is $\leq n-2$.

Everywhere in this section we consider singular homology groups reduced in dimension zero with coefficients in a given group $G$ (if $G$ is not written then the coefficients are integers). The following notion introduced by Toru\'{n}czyk \cite{to78} is well known:
A closed subset $A$ of a space $X$ is said to be a {\em  $Z_n$-set in $X$} if for any map $f\colon\mathbb I^n\to X$ and any open cover $\omega$ of $X$ there is a map $g\colon\mathbb I^n\to X\setminus A$ which is $\omega$-close to $f$ (i.e., for all $x\in\mathbb I^n$ both points $f(x)$ and $g(x)$ belong to some element of $\omega$). A homological counterpart of this notion was defined by
Banakh-Cauty-Karassev \cite{bck}: A closed subset $A\subset X$ is called a {\em $G$-homological $Z_n$-set in $X$} if $H_k(U,U\setminus A;G)=0$ for all $k\leq n$ and all open $U\subset X$. It follows from the excision axiom that a point $x\in X$ is a $G$-homological $Z_n$-set in $X$ provided
$H_k(X,X\setminus x;G)=0$ for all $k\leq n$.

Further, let us remind the definition of the {\em separating dimension $t(X)$} of a space $X$ introduced by Steinke \cite{st} and its transfinite extension $trt(X)$ given by Arenas-Chatyrko-Puertas \cite{acp}: $trt(X)=-1$ iff $X=\varnothing$; $trt(X)\leq\alpha$ for an ordinal $\alpha$ if any closed set $B\subset X$ containing at least two points can be separated by a closed set $P\subset B$ with $trt(P)<\alpha$. When $trt(X)$ is an integer, then $trt(X)=t(X)$. Moreover, $t(X)\leq\rm{ind}(X)$ \cite{st}. Hence, for metrizable $X$ we have $t(X)\leq\rm{ind}(X)\leq\dim X$.

Following Mitchell \cite{mi}, we say that {\em $X$ has the property $H(n,G)$ at the points of a set $M\subset X$} if $H_k(X,X\setminus x;G)=0$ for all $k\leq n$ and all $x\in M$. When $M=X$ in the above definition, $X$ is said to have the {\em $H(n,G)$-property}.

\begin{thm}
Let $X$ be a complete metric space and $M$ be an $F_\sigma$-subset of $X$ such that $trt(M)\leq n-2$ and $X$ has the property $H(n-1,G)$ at the points of $M$. Suppose $P, Q\subset X$ are open sets which can be joined by an arc in $X$. Then there is an arc in $X\setminus M$ joining $P$ and $Q$. In particular, any arcwise connected open subset of $X$ is a Mazurkiewicz arc $n$-manifold provided $X$ has the property $H(n-1,G)$.
\end{thm}

\begin{proof}
Let $M$ be a countable union of closed sets $M_i$, $i=1,2,..$. Since each $x\in M$ is a $G$-homological $Z_{n-1}$-point in $X$ and $trt(M_i)\leq n-2$, we can apply Theorem 4.3 from \cite{bck} stating that if $A$ is a closed subset of a space $X$ such that $trt(A)=m$ and all $a\in A$ are $G$-homological $Z_{n+m}$-points in $X$, then $A$ is a $G$-homological $Z_{n}$-set in $X$ (for the covering dimension $\dim$ this was established by Daverman \cite[Lemma 2.1]{da}). Therefore, any one of the sets $M_i$ is a $G$-homological $Z_1$-set in $X$. Then, by \cite[Theorem 3.2(6)]{bck}, $M_i$ are $Z_1$-sets in $X$. Consequently, the spaces $C_i=C(\mathbb I,X\setminus M_i)$ are dense (and obviously, open) in the space $C(\mathbb I,X)$ of all continuous maps from $\mathbb I$ into $X$ equipped with the compact-open topology. Finally, since $C(\mathbb I,X)$
is complete, $\bigcap C_i$ is also dense in $C(\mathbb I,X)$. Because $P$ and $Q$ can be joined by an arc in $X$, there exists a map $f\colon\mathbb I\to X$ with $f(0)\in P$ and $f(1)\in Q$. Since $P$ and $Q$ are open, $f$ can be approximated by maps $g\colon\mathbb I\to X\setminus M$ with $g(0)\in P$ and $g(1)\in Q$.

The second half follows from the first one and the following observations: if $U\subset X$ is open, then the excision axiom implies $H_k(U,U\setminus x;G)=0$ for all $k\leq n-1$ and $x\in U$; moreover, $trt(M)\leq\dim M$ for any set $M\subset X$.
\end{proof}

Below, by a {\em homology $n$-manifold over $G$} we mean a metric space $X$ such that for every $x\in X$ we have $H_k(X,X\setminus x;G)=0$ if $k\neq n$ and $H_n(X,X\setminus x;G)=G$.

\begin{cor}
Let $X$ be an arcwise connected complete metric space. In each of the following cases any arcwise connected open subset of $X$ is a Mazurkiewicz arc $n$-manifold:
\begin{itemize}
\item[(1)] $X$ is a homology $n$-manifold over a group $G$;
\item[(2)] $X$ is a product of at least $n$  metric spaces $X_i$, $1\leq i\leq m$.
\end{itemize}
\end{cor}

\begin{proof}
The first item follows directly from Theorem 4.1. To prove the second one, consider the exact sequence for every $i$ and $x\in X_i$
\[
\begin{diagram}
H_0(X_i\setminus x;G)\rightarrow H_0(X_i;G)\rightarrow H_0(X_i,X_i\setminus x;G)\rightarrow 0.
\end{diagram}
\] Since $X_i$ is arcwise connected (as an image of $X$), $H_0(X_i;G)=0$. So, $H_0(X_i,X_i\setminus x;G)=0$.
Then, by \cite[Theorem 3.2 and Theorem 6.1(2)]{bck}, $H_k(X,X\setminus x;G)=0$ for every
$x\in X$ and $k\leq m-1$. Finally, Theorem 4.1 completes the proof.
\end{proof}

In some situations the space $X\setminus M$ from Theorem 4.1 is arcwise connected.

\begin{thm}
Let $X\in H(n-1,\mathbb Z)$ be a connected and locally connected complete metric space and $M$ be an $F_\sigma$-set in $X$ with $\dim M\leq n-2$. Then $U\setminus M$ is arcwise connected for any open connected set $U\subset X$ with $U\setminus M\neq\varnothing$.
\end{thm}

\begin{proof}
Since $X$ is connected and locally connected, $U$ is arcwise connected  and locally arcwise connected. Moreover, by the excision axiom, $U$ has the property $H(n-1,\mathbb Z)$.
Let $M=\bigcup M_i$ with each $M_i$ being closed in $X$. It follows from the proof of Theorem 4.1 that each
$M_i\cap U$ is a $Z_1$-set in $U$. For any two points $a,b\in U\setminus M$ consider the set-valued map $\Phi\colon\mathbb I\to U$ defined by
$\Phi(0)=a$, $\Phi(1)=b$ and $\Phi(t)=U$ for each $t\in (0,1)$. It is easily seen that $\Phi$ is lower semi-continuous and for every $x\in U$ and its neighborhood $W$ there is a neighborhood $V\subset W$ of $x$ with the following property: if $t\in [0,1]$ and $x_1,x_2$ are two points from $\Phi(t)\cap V$ there is an arc in $\Phi(t)\cap W$ joining $x_1$ and $x_2$. Then, by \cite[Theorem 1.2]{gv}, $\Phi$ admits a continuous selection
$g\colon\mathbb I\to U\setminus M$. Obviously, $g(\mathbb I)$ is an arc in $U\setminus M$ joining $a$ and $b$.
\end{proof}

Next proposition shows that any space with a base consisting of Cantor manifolds is a Mazurkiewicz manifold.

\begin{pro}
Suppose $X$ is a connected complete metric space possessing a base $\mathcal B$ of open sets such that $\overline{U}$ is a Cantor manifold with respect to a given class $\mathcal C$ for every $U\in\mathcal B$. Then $X\setminus M$ is arcwise connected for every $F_\sigma$-subset $M=\bigcup M_i$ of $X$ with  $M_i\in\mathcal C$, $i\geq 1$.
\end{pro}

\begin{proof}
 According to a result of Had\v{z}iivanov-Hamamd\v{z}iev \cite[Theorem 1]{hh}, $X\setminus M$ is connected. Moreover, by the same result, all $U\setminus M$, $U\in\mathcal B$, are connected. Hence, $X\setminus M$ is connected and locally connected. Because $X\setminus M$ is complete (as a $G_\delta$-subset of $X$), it is arcwise connected.
\end{proof}

A metric space $X$ is said to have the {\em local separation property in dimension $n$} (written $LS^n$) if for every $x\in X$ and every neighborhood $U$ of $x$ there exists another neighborhood $V\subset U$ of $x$ such that any map
$f\colon\mathbb S^k\to V$, $k\leq n$, can be approximated by maps $g\colon\mathbb S^n\to V$ such that each $g(\mathbb S^k)$ does not separate $V$. It can be shown that if a space has the property $LS^n$, then the statement in above definition holds with $\mathbb S^k$
replaced by $\mathbb I^k$.

The {\em disjoint $(n,m)$-cell property} of a metric space $X$, denoted by $D(n,m)$, is defined as follows: for each $\epsilon>0$ and any two maps
$f\colon\mathbb I^n\to X$ and $g\colon\mathbb I^m\to X$ there exist maps $f'\colon\mathbb I^n\to X$ and $g'\colon\mathbb I^m\to X$ such that $f'$ and $g'$
are $\epsilon$-close to $f$ and $g$, respectively, and $f'(\mathbb I^n)\cap g'(\mathbb I^m)=\varnothing$. We are interested in the property $D(0,n)$, which implies the property $H(n,G)$ for any $G$.
Using an idea from Krupski \cite[Theorem 2.6]{kru93} we provide some conditions for a homogeneous spaces to have the property $H(n-1,G)$.

\begin{thm}
Any homogeneous locally compact metric $ANR$-space $X$ with $X\in LS^{n-2}$ has the $D(0,n-1)$-property. Thus $X$ has the $H(n-1,G)$-property for any group $G$.
\end{thm}

\begin{proof}
We are going to prove that $X$ has the $D(0,n-1)$-property. So, by \cite[Theorem 2.5]{kru93}, $X\in H(n-1,\mathbb Z)$ which, according to the Universal Coefficients Formula, implies $X\in H(n-1,G)$ for any group $G$. By Proposition 1.8 and Theorem 2.5 from \cite{kru93}, it suffices to show that every $a\in X$ is $LCC^{n-2}$ in $X$, i.e. for every neighborhood $U$ of $a$ there exists another neighborhood $V\subset U$ of $a$ such that any map
$f\colon\mathbb S^{k}\to V\backslash a$, $k\leq n-2$, can be extended to a map $\bar{f}\colon\mathbb B^{k+1}\to U\backslash a$ (here $\mathbb B^{k+1}$ is the $(k+1)$-dimensional ball). To this end, let $a\in U\subset X$ be an open connected set with a compact closure. Take an open neighborhood $V\subset U$ of $a$ such that $V$ is contractible in $U$. We may also suppose that $V$ satisfies the requirements from the definition of the
$LS^{n-2}$-property.

Let $f\colon\mathbb S^{k}\to V\backslash a$ be a map with $k\leq n-2$ and $g\colon\mathbb S^{k}\to V\backslash a$ be an approximation of $f$ such that $g(\mathbb S^{k})$ does not separate $V$. The proof will be done if $g$ can be extended to a map from $\mathbb B^{k+1}$ into $U\backslash a$. Indeed, $U\backslash a\in ANR$ implies that if $g$ is close enough to $f$, then $f$ and $g$ are homotopic in $U\backslash a$ and $f$ is extendable to a map $\overline{f}\colon\mathbb B^{k+1}\to U\backslash a$ provided $g$ has such an extension. According to the choice of $V$, there exists an extension $\overline{g}\colon\mathbb B^{k+1}\to U$ of $g$, and let $b\in U\backslash\overline{g}(\mathbb B^{k+1})$. Since  $g(\mathbb S^{k})$ does not separate $V$, it does not separate $U$. Then $U\backslash g(\mathbb S^{k})$ is connected and locally connected, so there exists an arc $C$ in   $U\backslash g(\mathbb S^{k})$ joining the points $a$ and $b$. Following an idea from the proof of \cite[Theorem 2.6]{kru93}, consider the set $$A=\{x\in C:\hbox{ }\mbox{there exists a map}\hbox{~} g_x\colon\mathbb B^{k+1}\to U\backslash x\hbox{~}\mbox{extending}\hbox{~}g\}.$$ Obviously, $A$ is open in $C$ and $b\in A$.  We are going to show that $A$ is also closed in $C$. That would imply that $A=C$ and $a\in A$, which will complete the proof.

Suppose $x\in\overline A$, and let $\epsilon<d(x,X\backslash U)/2$ be a positive number, where $d$ is a metric on $X$, such that if a map $g'\colon\mathbb S^{k}\to U\backslash x$ is $\epsilon$-close to $g$, then $g$ and $g'$ are homotopic in $U\backslash x$ (such $\epsilon$ exists because $U\backslash x$ is an $ANR$).
 Proceeding as in the proof of Theorem 2.6 from \cite{kru93} (with $\mathbb S^1$ replaced by $\mathbb S^{k}$), we can find $\delta>0$, a point
 $y\in A$ with $d(x,y)<\delta$ and a map $h\colon U\to U$ such that $h^{-1}(x)=y$ and $h$ is $\epsilon$-close to the identity on $U$. Since the map
 $g_y\colon\mathbb B^{k+1}\to U\backslash y$ (see the definition of $A$) extends $g$, then $h\circ g_y$ maps $\mathbb B^{k+1}$ into $U\backslash x$
and $(h\circ g_y)|\mathbb S^{k}=h\circ g$ is homotopic to $g$ in $U\backslash x$. Hence, by the homotopy extension property, $g$ can be extended to a map from $\mathbb B^{k+1}$ into $U\backslash x$. So, $x\in A$, which means that $A$ is closed.
\end{proof}

Recall the property $\triangle(n)$ of Borsuk \cite{bo}: $X\in\triangle(n)$ if for every $x\in X$ every neighborhood $U$ of $x$ contains a neighborhood $V$ of $x$ such that each compact nonempty set $B\subset V$ of dimension $\dim B\leq n-1$ is contractible in a subset of $U$ of dimension $\leq\dim B+1$. If $U$ in that definition has a compact closure, then $X\in\triangle(n)$ implies that every map $f\colon K\to U$, where $K$ is a compactum of dimension $\dim K\leq n$,  can be approximated  by maps $g\colon K\to U$ such that $\dim g(K)\leq\dim K$. On the other hand, if $X$ is a homogeneous locally compact, locally connected metric space of dimension $\dim\geq n$, then every region in $X$ is a Cantor $n$-manifold, see \cite{kv}.
Moreover, according to \cite[Observation 3.1]{kru93}, $X\times\mathbb R$ has the disjoint disk property provided $X$ is homogeneous locally compact $ANR$ of dimension $\geq 4$ satisfying $\triangle(2)$. Therefore, Theorem 4.5 implies next corollary which improves \cite[Note added in proof]{mi}.

\begin{cor}
Let $X$ be a homogeneous locally compact metric $ANR$-space such that $\dim X=n\geq 4$ and $X\in\triangle(n-2)$.
Then $X$ has the property $H(n-1,\mathbb Z)$ and the product $X\times\mathbb R$ has the disjoint disk property.
\end{cor}

It is still unknown whether the dimension of a product of two homogeneous $ANR$-compacta satisfies the logarithmic law. The next corollary provides a partial answer of this question.

\begin{cor}
Let $X\in LS^{n-2}$ be a homogeneous locally compact metric $ANR$-space such that $\dim X=n$.
Then $\dim X\times Y=\dim X+\dim Y$ for every compact metric space $Y$.
\end{cor}

\begin{proof}
 According to the proof of Theorem 4.5, $X\in LCC^{n-2}$. Using the terminology of Kodama's paper \cite{ko}, this means that every $x\in X$ is an $HL^{n-2}$-point. On the other hand, $X$ does not have
any $HL^{n-1}$-point. Indeed, otherwise each point of $X$ would be $HL^{n-1}$ by the homogeneity. Then $X$ has $D(0,n)$ by \cite[Proposition 1.8]{kru93} and,
by \cite[Corollary 2.4]{kru93}, $\dim X>n$, a contradiction. Therefore, we can apply
\cite[Corollary 2]{ko} to conclude that $\dim X\times Y=\dim X+\dim Y$ for any compact metric space $Y$.
\end{proof}
%%%%%%%%%% Bibliography %%%%%%%%%%%%%%%%%%%%%%%%%%

\end{document}